\documentclass{article}
\usepackage{graphicx} 
\usepackage{amsmath, amssymb, amsthm}
\usepackage{enumerate}
\newtheorem{thm}{Theorem}[section]
\newtheorem{cor}[thm]{Corollary}
\newtheorem{lem}[thm]{Lemma}

\theoremstyle{definition}

\title{Cardinalities of Ultraproducts of Finite Sets in ZF + DC}
\author{Jacob Kowalczyk}
\date{November 10, 2026}

\begin{document}

\maketitle

\begin{abstract}
    Relative to an inaccessible cardinal, it is consistent with ZF + DC that for some ultrafilter $U$ on $\omega$, two infinite ultraproducts of finite sets $\prod A_n / U$ and $\prod B_n / U$ have the same cardinality if and only if $0 < \lim_U |A_n|/|B_n| < \infty$. In particular, this holds in $W[U]$, where $W$ is the Solovay Model and  $U$ is $[\omega]^\omega$-generic.
\end{abstract}

\section{Introduction}

It is not known if it is consistent with ZFC to have an ultrafilter $U$ on $\omega$ with a pair of non-isomorphic infinite ultraproducts of finite linear orders, equivalently, whether there is a pair of non-isomorphic infinite initial segments of $\omega^\omega /U$ with maximal elements. This question arose from the work of Bartošová, Džamonja, Patel, and Scow in \cite{bartošová2025bigramseydegreesultraproducts}. In this paper we resolve the question in the ZF + DC case and determine which pairs of these ultraproducts are provably isomorphic in ZF + DC. Furthermore, we show that a pair of these ultraproducts is provably isomorphic in ZF + DC if and only if they are provably of the same cardinality.

The model of ZF + DC we are considering in this paper is $W[U]$, where $W$ is the Solovay Model and $U$ is generic over $[\omega]^\omega$. It is fairly straightforward to prove in ZF + DC that two ultraproducts of finite linear orders are isomorphic if the sizes of the linear orders grow at a similar rate. The core of the paper is proving that in $W[U]$ two ultraproducts of finite sets have different cardinalities when the sizes of the sets do not grow at similar rates. The results of the paper are summarized in the following theorem.
%
%

\begin{thm}
    Relative to an inaccessible cardinal, it is consistent with ZF + DC that there is an ultrafilter $U$ on $\omega$ such that for all $f,g: \omega \rightarrow \omega$ with $f(n),g(n) \rightarrow_U \infty$, the following are equivalent:
    \begin{enumerate}[{(1)}]
        \item $0 < \lim_U f(n)/g(n) < \infty$
        \item $\prod_n (f(n),<)/U \cong \prod_n (g(n),<)/U$
        \item $\prod_n f(n)/U$ and $\prod_n g(n)/U$ have the same cardinality.
    \end{enumerate}
    The implications $(1) \implies (2) \implies (3)$ are provable in ZF + DC.
\end{thm}


\section{Background and Notation}

We let $<_U$ denote the linear order given by the ultraproduct $\prod_n (\omega, <) / U$. Note that $\prod_n (f(n),<) / U$ is isomorphic to the initial segment of $\prod_n (\omega, <) / U$ with maximal element $[f]$.

For an ultrafilter $U$ on $\omega$, and a sequence $(x_n)_{n < \omega}$ of reals, we define $\lim_U x_n$ to be the unique $x \in \mathbb{R} \cup \{\pm \infty\}$ such that for all open neighborhoods $O$ of $x$, $\{n : x_n \in O\} \in U$. It is provable in ZF + DC that the limit $x$ always exists and is unique, proofs can be found in many texts, for example \cite{HindmanStrauss+2012}, just note that no choice beyond DC is needed.

Fix an inaccessible cardinal $\kappa$ and let $G \subseteq \operatorname{Coll}(\omega, <\kappa)$ be $V$-generic. The inner model $W = \operatorname{HOD}^{V[G]}_{V \cup \mathbb{R}}$ is the Solovay model we will be referring to in this paper. Basic information on the Solovay model can be found in many set theory textbooks, including \cite{jech}. 

The poset $[\omega]^\omega$, ordered by reverse almost inclusion, adds a Ramsey ultrafilter and is \textit{balanced} over $W$. Its \textit{balanced conditions} are ultrafilters. The theory of balanced forcings over the Solovay model is studied in \cite{GST}. In this paper, we only need the following lemma: it is an immediate consequence of \cite[Proposition 5.2.2]{GST} and \cite[Theorem 7.1.2]{GST}.

\begin{lem}
\label{balance}
    Let $W$ be the Solovay model. If $V[K]$ is an intermediate model and $F \in V[K]$ is an ultrafilter, then when viewed as a condition in $[\omega]^\omega$, $F$ decides all statements in the forcing language of $[\omega]^\omega$ in $V[K]$.
\end{lem}

In this paper we are concerned with the model $W[U]$, where $U$ is generic for the poset $[\omega]^\omega$. This is one of two natural choiceless models used to study a Ramsey ultrafilter, the other being $L(\mathbb{R})[U]$, where $L(\mathbb{R})$ is a model of $AD$. Both of these models have been studied in the past (see \cite{Barren}, \cite{PSP}). A notable result of Todorcevic says that under sufficient large cardinal assumptions, every Ramsey ultrafilter is $L(\mathbb{R})$ generic over $[\omega]^\omega$ \cite{Farah}. We do not know if our characterization of the cardinalities of ultraproducts of finite sets in $W[U]$ also holds in the model $L(\mathbb{R})[U]$.

\section{Proof of the Main Theorem}
Proving from ZF + DC that $\prod_n (f(n),<)/ U$ and $\prod_n (g(n),<) / U$ are isomorphic when $0 < \lim_U f(n)/g(n) < \infty$ is relatively simple. Note that this holds for any ultrafilter $U$.

\begin{thm}[ZF + DC] If $U$ is an ultrafilter, $f,g: \omega \rightarrow \omega$ are such that $0 < \lim_U \frac{f(n)}{g(n)} < \infty$, and $\prod_n(f(n),<) / U$ and $\prod_n(g(n),<) / U$ are infinite, then $\prod_n(f(n),<) / U \cong \prod_n(g(n),<) / U$.
\end{thm}

\begin{proof}
Fix an ultrafilter $U$. For $f<_U g$, define $m_g(f) = \lim_U \dfrac{f(n)}{g(n)}$.
Note that $m_g$ is a surjective weakly order preserving function onto $[0,1]$, and that if $a <_U b$, then $m_g(a) = m_g(b)$ iff $m_g(b-a) = 0$. Furthermore, if $f<_U g$ is such that $m_g(f) = r > 0$, then $m_g(a) = rm_f(a)$ for all $a <_U f$. This is because $rm_f(a) = \lim_U r\dfrac{a(n)}{f(n)} = \lim_U \dfrac{f(n)}{g(n)} \dfrac{a(n)}{f(n)} = \lim_U \dfrac{a(n)}{g(n)} = m_g(a)$. In particular $m_f(a) = 0$ iff $m_g(a) = 0$.

The function $m_g$ partitions $\{h \leq_U g\}$ into equivalence classes, with one class $E_{g,x}$ for each $x \in [0,1]$. Two functions $a \leq_U b$ are in the same equivalence class if and only if $m_g(b-a) = 0$. So for any $0 < x < 1$ and $a \in E_{g,x}$, we have $E_{g,x} = \{a -c : c \in E_{g,0}\} \cup \{a + c: c \in E_{g,0}\}$. Similarly, $E_{g,1} = E_{g,0}^*$. Therefore
\begin{equation*}
    \{h \leq_U g\} \cong E_{g,0} + (0,1)\times(E_{g,0}^* + E_{g,0}) + E_{g,0}^*,
\end{equation*}
where the product has the lexicographic order.

Since $m_g(f) = r >0$, we get $E_{g,0} = E_{f,0}$ and the result follows. In particular, any order preserving bijection $k: [0,r] \rightarrow [0,1]$, along with choices of representatives from $E_{f,x}$ and $E_{g,k(x)}$ for $0 < x <1$ gives an isomorphism. One choice of representative of $E_{f,x}$ is $\lceil xf \rceil$.
\end{proof}

To prove that the ultraproducts have different cardinalities in $W[U]$, we introduce the poset $\mathbb{P}_{F,f,g}$ for an ultrafilter $F$ on $\omega$ and $f>_F g$. This poset adds a real in a controlled way, which when combined with basic geometric set theory methods, allows us to break any possible candidate for an injective function $\prod_n f(n)/U \rightarrow \prod_n g(n)/ U$.
Given an ultrafilter $F$ and $f,g :\omega \rightarrow \omega$ with $\lim_F \frac{f(n)}{g(n)} = \infty$, we define the poset
\begin{equation*}
    \mathbb{P}_{F,f,g} = \left\{p: \operatorname{dom}(p) = \omega, p(n) \subseteq (f(n)+1), p(n) \neq \emptyset, \lim_F \frac{|p(n)|}{g(n)} = \infty\right\},
\end{equation*}
with $p \leq q$ if $p(n) \subseteq q(n)$ for all $n$. For a condition $p \in \mathbb{P}_{F,f,g}$, we define $[p]:= \prod_n p(n)$, and for $A \subseteq \omega$ define $[p]\upharpoonright A := \prod_{n \in A}p(n)$. This poset adds a generic real $x_{gen} = \bigcap \{ [p]: p \in G\}$, where $G$ is the generic filter on $\mathbb{P}_{F,f,g}$.

Useful properties of $\mathbb{P}_{F,f,g}$ are summarized in the following theorem. We use the phrase \textit{Lipschitz}$^*$ \textit{reading of names} to refer to the property in corollary 3.5.
\begin{thm}
    $\mathbb{P}_{F,f,g}$ adds a new $F$-class. If $F$ is a P-point, then $\mathbb{P}_{F,f,g}$ is proper and has continuous reading of names. If $F$ is a rapid P-point, then $\mathbb{P}_{F,f,g}$ adds no independent reals. If $F$ is Ramsey, then $\mathbb{P}_{F,f,g}$ has Lipschitz$^*$ reading of names and preserves the ultrafilter $F$. 
\end{thm}

Given $\mathbb{P}_{F,f,g}$, let $\mathcal{G}_P$ demote the game where player I plays $p_i \in \mathbb{P}_{U,f,g}$, player II plays $b_i \in [\omega]^{<\omega}$, and player I is required to play so that $p_j \upharpoonright b_i = p_i \upharpoonright b_i$ for all $j > i$. Player I wins if the sequence $p_0 \geq p_1 \geq \cdots$ has no lower bound.

Let $\mathcal{G}_R$ denote the game where the $b_i$ are required to be singletons. These games are simple variations of the standard P-Point and Ramsey games on an ultrafilter, and they allow us to give simple proofs of both the properness and the continuous reading of names of $\mathbb{P}_{F,f,g}$.

\begin{lem}
    If $F$ is a P-point (Ramsey), then player I has no winning strategy in $\mathcal{G}_P$ ($\mathcal{G}_R$).
\end{lem}

\begin{proof}
    We show the P-point case, the Ramsey case is similar.
    Recall the P-point game $\mathcal{G}'_P$ associated to $F$, where player I plays decreasing $X_i \in F$, player II plays $b_i \in [X_i]^{< \omega}$, and player II wins if $\bigcup b_i \in F$. Player I has a winning strategy if and only if $F$ is not a P-point.
    
    Suppose $F$ is a P-point and let $\sigma$ be a strategy for player I in $\mathcal{G}_P$. Define $\sigma'$, a $\mathcal{G_P'}$ player I strategy as follows. Given $\sigma(b_0, b_1, \dots, b_{n-1}) = p_n$, let $\sigma'(b_0, b_1, \dots, b_{n-1}) = \{k: \dfrac{|p_n(k)|}{g(k)} > n\}$. Since $F$ is a P-point, player II has a play $(b_i)_{i<\omega}$ which beats $\sigma'$. Let $p_0 \geq p_1 \geq \cdots$ be the sequence produced when $(b_i)_{i<\omega}$ is played against $\sigma$. Define $p_{\infty}$ by $p_\infty \upharpoonright b_i = p_i \upharpoonright b_i$ for all $i$, and for $n \notin \bigcup b_i$, $p_\infty (n)$ is the least singleton such that $p_\infty(n) \subseteq p_i(n)$ for all $n$. If $k \in \bigcup_{i>m} b_i$, then $\dfrac{|p_\infty(k)|}{g(k)} > m$, and $\bigcup_{i>m} b_i \in F$ since $(b_i)_{i < \omega}$ beats $\sigma'$, so $p_\infty \in \mathbb{P}_{F,f,g}$, and $(b_i)_{i < \omega}$ beats $\sigma$ as well. Therefore $\sigma$ is not winning.
\end{proof}

\begin{cor}
    If $F$ is a P-point, then $\mathbb{P}_{F,f,g}$ is proper.
\end{cor}
\begin{proof}
    Let $M \prec H(\theta)$ be countable elementary and let $p \in M$. Let $(D_i)_{i < \omega}$ be an enumeration of the open dense subsets of $\mathbb{P}_{F,f,g}$ in $M$.
    
    We construct a $\mathcal{G}_P$ strategy $\sigma$ as follows. Let $p_s$ denote $\sigma(s)$. Let $p_{\emptyset} \leq p$ be such that $p_\emptyset \in D_0 \cap M$. Suppose we have $s = (b_0,\dots, b_n)$ and $p_{s\upharpoonright n}$. Enumerate the finitely many $t \in \omega^{(b_0 \cup \cdots \cup b_n)}$ such that $t(i) \in p_{s \upharpoonright n}(i)$ as $t_0, \dots, t_k$. Let $p_{s \upharpoonright n} \geq p^{t_0}_s \geq p^{t_1}_s \geq \cdots \geq p^{t_k}_s$ be such that $p^{t_j}_s$ agrees with $p_{s \upharpoonright n}$ on $b_0 \cup \cdots \cup b_n$ and if $p^{t_j}_s$ is strengthened to agree with $t_j$ on $b_0 \cup \cdots \cup b_n$, then the resulting condition is in $D_n \cap M$. Then let $p_s = p_{s}^{t_k}$.
    Since $\sigma$ is not winning, there is a play $(b_i)$ which beats it and produces the condition $p_{\infty}$, which is clearly a master condition for $M$.
\end{proof}


\begin{cor}
    If $F$ is a P-point, then $\mathbb{P}_{F,f,g}$ has continuous reading of names, that is if $p \Vdash \tau \in \omega^\omega$, then there exists a condition $q\leq p$ and a continuous function $h:[q] \rightarrow \omega^\omega$ such that $q \Vdash h(\dot{x}_{gen}) = \tau$.
    
    Furthermore, if $F$ is Ramsey, then we can take these $q$ and $h$ to be such that there exists some $B \in F$ such that $h \upharpoonright B: [q] \upharpoonright B \rightarrow \omega^B$ is Lipschitz.
\end{cor}

\begin{proof}
    Let $p \Vdash \tau \in \omega^\omega$. By repeating the proof of properness with $D_i$ being the set of conditions that decide $\tau(i)$, we get a sequence of finite sets $(b_i)_{i<\omega}$ and a condition $p_{\infty} \leq p$ such that deciding $p_\infty$ on $\bigcup_{i<k} b_i$ decides $\tau(k)$. So we can define $h:[p_\infty] \rightarrow \omega^\omega$ by letting $h(x)(n)$ be the value that $\tau(n)$ is forced to take by 
    conditions $r \leq p_\infty$ with $r(j) = \{x(j)\}$ for $j \in \bigcup_{i < n}b_i$.
    
    If $F$ is Ramsey, then note that we can take the $b_i$ to be an increasing sequence of singletons, and by slightly modifying the proof of properness we can ensure that deciding $p_\infty$ up to $b_{i+1}$ decides $\tau$ up to $b_i$. Then one of $\bigcup_{i<\omega} b_{2i}$ or $\bigcup_{i < \omega} b_{2i+1}$ must belong to the ultrafilter, call that set $B$. Then any $q \leq p_\infty$ which decides all values outside of $B$ with $h:[q] \rightarrow \omega^\omega$ defined as above works.
\end{proof}

We use the following theorem of Shelah and Zapeltal to prove the remaining desired properties of $\mathbb{P}_{F,f,g}$.

\begin{thm}[Shelah, Zapletal] (\cite{Sh:952}) Suppose $k \in \omega$ and $r_n : n \in \omega$ is a sequence of real numbers. Then for every sequence of submeasures $\phi_n : n \in \omega$ on finite sets, increasing fast enough, and for every partition $B_i: i \in k$ of the product $(\prod_n \operatorname{dom}(\phi_n)) \times \omega$ into Borel pieces, one of the pieces contains a product of the form $(\prod_n b_n) \times c$ where $c \subseteq \omega$ is an infinite set, and $b_n \subseteq \operatorname{dom}(\phi_n)$ and $\phi_n(b_n) > r_n$ for every number $n \in \omega$.
\end{thm}

\begin{lem}
    If $F$ is a rapid P-point, then $\mathbb{P}_{F,f,g}$ does not add independent reals.
\end{lem}

\begin{proof}
    Let $p \Vdash \tau \in 2^\omega$, we must show that there is a condition $q \leq p$ that decides infinitely many values of $\tau$.
    
    By continuous reading of names we may assume there is a continuous function $h:[p] \rightarrow 2^\omega$ with $p \Vdash h(\dot{x}_{gen}) = \tau$. Define submeasures $\phi_n$ on $p(n)$ by $\phi_n(a) = \dfrac{|a|}{g(n)}$. By the above theorem and since $F$ is rapid, there exists a subsequence $\phi_{n_j}$ with $\{n_j\} \in F$ which grows fast enough that for every partition of $(\prod_{n_j} p(n_j)) \times \omega$ into Borel sets $B_0,B_1$, one of the pieces contains a product of the form $(\prod_{n_j} b_{n_j}) \times c$ where $c \subseteq \omega$ is infinite, $b_{n_j} \subseteq p(n_j)$, and $\phi_{n_j}(b_{n_j}) > j$ for each $j \in \omega$.
    
    Strengthen $p$ to decide all its values outside of $\{n_j\}$. This allows us to view $h$ as a function $h: \prod_{n_j} b_{n_j} \rightarrow 2^\omega$. Let $B_0 = \{(x,j): h(x)(n_j) = 0\}$ and $B_1 = \{(x,j): h(x)(n_j) = 1\}$. One of the $B_i$ contains a product of the form $(\prod_{n_j} b_{n_j}) \times c$ with $c \subseteq \omega$ infinite. Letting $q\leq p$ be such that $q(n_j) = b_{n_j}$, we get a condition in $\mathbb{P}_{U,f,g}$ which forces that $h \upharpoonright [q]$ is decided on all $n_j$ with $j\in c$, thus $q$ decides $\tau$ on an infinite set.
\end{proof}

\begin{lem}
    If $F$ is Ramsey, then $\mathbb{P}_{F,f,g}$ preserves the ultrafilter $F$.
\end{lem}

\begin{proof}
    %
    Let $p \Vdash \tau \in 2^\omega$, we must show there exists some $A \in F$ and $q \in \mathbb{P}_{F,f,g}$ such that either $q \Vdash A \subseteq \tau$ or $q \Vdash A \cap \tau = \emptyset$.
    
    Since $F$ is Ramsey, it intersects every $\Sigma^{1}_{1}$ dense set in $[\omega]^\omega$ (\cite{FI} 3.4.3). By strengthening $p$ if necessary, we may assume there is a continuous $h:[p] \rightarrow 2^\omega$ with $p \Vdash h(\dot{x}_{gen})= \tau$. Consider the following set
    \begin{align*}
        D = \big\{A &\subseteq \omega: \exists q\ \operatorname{dom}(q) = \omega,\ q(n) \subseteq p(n),\ \lim_{n \in A}\dfrac{|q(n)|}{g(n)} = \infty,\\
        &[(\forall n \in A \ \forall x \in [q], \ h(x)(n)=0 ) \vee  (\forall n \in A \ \forall x \in [q], \ h(x)(n)=1 )]\big\},
    \end{align*}
    
    This set is clearly $\Sigma^{1}_{1}$. Suppose we have shown it is dense, and let $A \in F \cap D$. By the definition of the set, there is some $q\in \mathbb{P}_{F,f,g}$, $q \leq p$ such that either $q \Vdash \tau \cap A = \emptyset$ or $q \Vdash A \subseteq \tau$, as desired. Therefore it suffices to show this set is dense.
    
    Let $B \subseteq \omega$.
    Repeating the proof of the last lemma, where $\{n_j\} \subseteq B$ is any infinite set on which the submeasures grow fast enough, gives us some $B' \subseteq \{n_j\} \subseteq B$ and a function $q$ (not necessarily in $\mathbb{P}_{F,f,g}$) such that $\operatorname{dom}(q) = \omega$, $q(n) \subseteq p(n)$, $\lim_{n \in B'}\dfrac{|q(n)|}{g(n)} = \infty$, and $h\upharpoonright [q]$ is decided on $B'$. Now let $q' \leq q$ agree with $q$ on $B'$ and decide all other values. Clearly we get that $B' \in D$, witnessed by $q'$. Therefore $D$ is dense.
\end{proof}

Lemma 3.8 completes the proof of Theorem 3.2. With the properties of $\mathbb{P}_{F,f,g}$ listed in 3.2, we are able to complete the proof of the main theorem.

\begin{thm}
    Let $W$ be the Solovay model and $U$ be $[\omega]^\omega$-generic over $W$. In $W[U]$ if $f$ and $g$ are such that $\lim_U \dfrac{f(n)}{g(n)} = \infty$, then there is no injection $\varphi: \prod_n f(n) / U \rightarrow \prod_n g(n) / U$.
\end{thm}

\begin{proof}
 In $W$, let $p \in [\omega]^\omega$ be a condition and $\tau$ a name such that $p \Vdash \tau: \prod_n f(n)/U \rightarrow \prod_n g(n)/U$ is an injection. Let $V[K]$ be a model obtained by forcing with a small poset over $V$ such that $p,\tau \in V[K]$. Let $u$ be $[\omega]^\omega$ generic over $V[K]$ containing $p$. Let $\mathbb{P} = \mathbb{P}_{u,f,g}$.

 Since $\mathbb{P}$ preserves $u$, it remains an ultrafilter, and thus a balanced condition, after forcing with $\mathbb{P}$. Therefore by Lemma \ref{balance} applied to the intermediate model $V[K][u][x_{gen}]$, where $x_{gen}$ is $\mathbb{P}$-generic, there is some real $r \in V[K][u][x_{gen}]$ such that $u$ forces $r$ is in $\tau([x_{gen}])$, since if we fix some $[\omega]^\omega$ name $\dot{r}$ for a real in $\tau([x_{gen}])$, Lemma \ref{balance} ensures that all its values are already decided by $u$. Thus, in $V[K][u]$, there is a $\mathbb{P}$ name $\sigma$ for a representative of $\tau([x_{gen}])$, and by continuous reading of names, there are $q \in \mathbb{P}$, $h:[q] \rightarrow \omega^\omega$ continuous, and $B \in u$, such that $q \Vdash u \Vdash h(x_{gen}) = \sigma$ and $h\upharpoonright B: [q] \upharpoonright B \rightarrow \omega^B$ is Lipschitz, and $q$ is decided outside of $B$.
 
 Let $G$ be the $\operatorname{Coll}(\omega, < \kappa)$ generic used in the construction of the Solovay model $W$. In $V[G]$, the set $u$ is countable, and there are only countably many dense subsets of $\mathbb{P}$. Enumerate $u$ as $(A_i)_{i < \omega}$ and the dense subsets of $\mathbb{P}$ as $(D_i)_{i < \omega}$. We can recursively build conditions $q_n \leq q$ and natural numbers $b_n \in B$ so that $q_{n+1} \leq q_n$, $q_j \upharpoonright \{b_0,\dots,b_i\} = q_i \upharpoonright \{b_0,\dots,b_i\}$ for $i < j$, $b_n \in \bigcap_{k \leq n} A_k$, $\dfrac{|q_n(b_n)|}{g(b_n)} > n$, and deciding $q_n$ on $\{b_0,\dots,b_n\}$ gives a condition in $D_n$. Form $q_\infty$ by taking $q_\infty(b_n) = q_n(b_n)$ for each $n$, and decide $q_\infty$ outside of $B' = \{b_0,b_1,\dots\}$. By construction, the set $B' \subseteq B$ diagonalizes the ultrafilter $u$, every $x \in [q_\infty]$ is generic over $\mathbb{P}$, and for each $n$, $|q_\infty(b_n)| > g(b_n)$. Note that $q_\infty$ is not a condition in $\mathbb{P}$.

 Since every $x \in [q_\infty]$ is $\mathbb{P}$ generic, $u$ forces that $h(x)$ is a representative of $\tau([x])$ for all $x \in [q_\infty]$. We recursively build $x,y \in [q_\infty]$ such that for all $n \in B'$, ($x(n) = y(n)$ implies $h(x)(n) \neq h(y)(n)$) and ($x(n) \neq y(n)$ implies $h(x)(n) = h(y)(n)$).
 
 Suppose $x$ and $y$ are defined up to $n \in B'$. Define $\psi_x,\psi_y : q_\infty (n) \rightarrow g(n)$ by setting $\psi_x(m)$ equal to $h(x)(n)$ when $x(n) = m$ and $\psi_y(m)$ equal to $h(y)(n)$ when $y(n) = m$. If we cannot satisfy the desired conditions, then $\psi_x(m) = \psi_y(m)$ for all $m$, and $\psi_x(m_1) \neq \psi_y(m_2)$ whenever $m_1 \neq m_2$. The first condition implies $\psi_x = \psi_y$, but then the second condition implies that this function is injective, which cannot happen since $|q_\infty(n)| > g(n)$.

  Let $x$ and $y$ be as above, let $C \subseteq B'$ be an infinite set on which one of the two conditions occur. Then $C \leq u$ and we either have $C \Vdash x=_U y \wedge h(x) \neq_U h(y)$ or vice versa, which contradicts the fact that $C \Vdash h(x) \in \tau([x]) \wedge h(y) \in \tau([y])$.
\end{proof}
This completes the proof of the main theorem.

\section*{Funding}

This research was partially supported by the National Science Foundation (NSF) under Grant DMS-1953955.

\bibliographystyle{plain}
\bibliography{bibliography.bib}
\end{document}